\documentclass[11pt, letterpaper]{article}   
\usepackage{hyperref} 
\usepackage{amssymb}
\usepackage{amsmath}
\usepackage{mathrsfs}
\usepackage{fullpage}
\usepackage{setspace}
\usepackage{enumitem}
\usepackage{amsthm}
\usepackage{mathtools}

\usepackage{tikz}
\tikzset{node distance=2cm, auto}
\usepackage{tikz-cd}
\usepackage{graphicx}
 
\newcommand\Z{\mathbb Z}

\newcommand\Aut{\mathrm{Aut}}
\newcommand\Inn{\mathrm{Inn}}
\newcommand\Out{\mathrm{Out}}
\newcommand\Hom{\mathrm{Hom}}
\newcommand\Res{\mathrm{Res}}
\newcommand\Ind{\mathrm{Ind}}
\newcommand\Inf{\mathrm{Inf}}
\newcommand\Def{\mathrm{Def}}
\newcommand\Iso{\mathrm{Iso}}
\newcommand\restr[2]{\left.\kern-\nulldelimiterspace#1\right|_{#2}}

\newtheorem{theorem}{Theorem}[section]

\newtheorem{proposition}[theorem]{Proposition}
\newtheorem{corollary}[theorem]{Corollary}

\newtheorem{example}[theorem]{Example}

\begin{document}

\begin{center}

\textbf{Burnside Rings of Fusion Systems and their Unit Groups}

Jamison Barsotti // Rob Carman\\
Department of Mathematics\\
William \& Mary\\
jbbarsotti@wm.edu // wrcarman@wm.edu\\

\end{center}

\vspace{.5cm}

\textbf{Abstract.}
For a saturated fusion system $\mathcal F$ on a $p$-group $S$,
we study the Burnside ring of the fusion system $B(\mathcal F)$,
as defined by Matthew Gelvin and Sune Reeh,
which is a subring of the Burnside ring $B(S)$.
We give criteria for an element of $B(S)$ to be in $B(\mathcal F)$
determined by the $\mathcal F$-automorphism groups of essential subgroups of $S$.
When $\mathcal F$ is the fusion system induced by a finite group $G$ with $S$ as a Sylow $p$-group,
we show that the restriction of $B(G)$ to $B(S)$ has image equal to $B(\mathcal F)$.
We also show that for $p=2,$
we can gain information about the fusion system
by studying the unit group $B(\mathcal F)^\times.$
When $S$ is abelian, we completely determine this unit group.

\vspace{.5cm}

\section{Introduction}
Throughout this article we let $\mathcal F$ denote a saturated fusion system on a $p$-group $S$.
Following the work of Sune Reeh in \cite{Reeh},
we define the Burnside ring of $\mathcal F$ as a certain subring $B(\mathcal F)\subseteq B(S)$.
In that article, he shows that $B(\mathcal F)$
is free abelian, just as $B(S)$ is.
Moreover, he finds a canonical basis for $B(\mathcal F)$,
showing that its rank is equal to the number of $\mathcal F$-conjugacy classes of subgroups of $S$.
So the Burnside ring of $\mathcal F$ can detect the number of $\mathcal F$-classes,
but not much else has been studied concerning the structure of this ring and how it relates to $\mathcal F$.
We give a few examples to show that $B(\mathcal F)$ cannot determine $S$ up to isomorphism,
nor can it detect the $\mathcal F$-automorphism groups.
Also when $\mathcal F\subseteq\mathcal F'$ are two fusion systems on $S$,
the subrings $B(\mathcal F')\subseteq B(\mathcal F')\subseteq B(S)$ 
do not necessarily preserve proper inclusions.
However, when $p=2,$
much more about $\mathcal F$ can be deduced from the structure of $B(\mathcal F)$.
Also in this case, the unit group $B(\mathcal F)^\times$
can be used to determine much information about $\mathcal F$ and vice versa.
We see this as a first attempt towards determining $B(\mathcal F)^\times$
for a fusion system on a 2-group.
When $S$ is abelian we are able to do so.

In Section \ref{bgBurnside},
we give the necessary background on Burnside rings of finite groups.
The ghost ring of a Burnside ring, which we will use to complete some useful computations,
is explained here as well.
Also we define the notion of bisets for two finite groups
and show they how can be used to create maps between the relevant Burnside rings.
Then in Section \ref{bgFusion},
we recall the basics of fusion systems.
We attempt to describe only the notions we will need to use in later proofs.
For a thorough introduction to fusions systems,
see Part I of \cite{AKO}.

In Section \ref{BFusion},
we define $B(\mathcal F),$ the Burnside ring of a fusion system as we will use it here.
We state some known results about $B(\mathcal F)$
and also give some examples to demonstrate some drawbacks of attempting to study
$\mathcal F$ via $B(\mathcal F)$.

We describe actions of the $\mathcal F$-outer automorphism groups on
the restrictions of $B(S)$ and corresponding ghost rings in Section \ref{OutAction}.
Theorem \ref{stablecheck} then gives a new criteria
for checking if an element of $B(S)$ is contained in $B(\mathcal F)$.
Then in Section \ref{Units} we begin to study the unit group $B(\mathcal F)^\times$,
which is a subgroup of $B(S)^\times$, an elementary abelian 2-group.
We show that the $\mathcal F$-outer automorphism group actions
induce actions on the restrictions of $B(S)^\times,$
and thus get criteria for a unit of $B(S)$ to be in $B(\mathcal F)^\times$.
We also define trace maps between particular fixed-point subgroups of $B(S)^\times$
and show that they are often split surjective.

We focus on Frobenius fusion systems in Section \ref{Frobenius},
and show that if $\mathcal F$ is determined by a finite group $G$ with Sylow $p$-subgroup $S$,
then $B(\mathcal F)$ is exactly the image of the restriction map from $B(G)$ to $B(S)$.
This is the content of Theorem \ref{Boltje}.
Combining this result with results from Serge Bouc on $B(S)^\times$,
we can show that if $S$ is a normal Sylow 2-subgroup of $G$,
then $B(\mathcal F)^\times\cong B(G)^\times$.
More generally, we consider also the fusion system of $N=N_G(S)$ on $S$.
We can then transfer questions about the unit groups of $B(N)$ and $B(G)$
to questions about the unit groups of Burnside rings of the fusions systems on $S$
induced by $N$ and $G$, respectively.

In Section \ref{MaxUnits},
we detail particular units of $B(S)^\times$
indexed by the maximal subgroups of $S$.
We show that $\Aut(S)$ permutes these units
and use this to determine when such units are contained in $B(\mathcal F)^\times.$
We show this can also be used to determine strong fusion properties of these maximal subgroups.
And when $S$ is abelian,
we show in section \ref{Abelian}
how to completely determine $B(\mathcal F)^\times.$
The rank of this elementary abelian group is proved in Theorem \ref{abrank}.

\subsection{Notation}
For a group $G$, and an element $x\in G,$
we define the conjugation (inner) automorphism defined by $x$ as
$c_x:G\longrightarrow G,g\mapsto xgx^{-1}$.
We denote the set of automorphisms of $G$ by $\Aut(G),$
the inner automorphisms by $\Inn(G)$,
and the outer automorphism group by $\Out(G)=\Aut(G)/\Inn(G)$.
For a prime $p$, we let $\mathrm{Syl}_p(G)$ denote the set of Sylow $p$-subgroups of $G$.
When $H$ is a subgroup (resp. a proper subgroup, resp. a normal subgroup) of $G$,
we denote this by $H\leq G$ (resp. $H<G$, resp. $H\unlhd G$).
We denote the conjugate subgroup $c_x(H)$ by $\prescript{x}{}{H}$ or $H^{x^{-1}}$.
When $H,K\leq G$,
we denote by $\Hom_G(H,K)$ the set of all (injective) group homomorphisms
$H\longrightarrow K$ that are defined by conjugation by a fixed element of $G$.
And $\Aut_G(H)=\Hom_G(H,H)$ contains $\Inn(H)$,
so we can define $\Out_G(H)=\Aut_G(H)/\Inn(H)$.

\section{Background on Burnside Rings}\label{bgBurnside}
Throughout this section, we fix a finite group $G$.
If $X$ and $Y$ are two finite $G$-sets,
then so are their disjoint union $X\sqcup Y$ and direct product $X\times Y$.
We define the Burnside ring $B(G)$ to be the Grothendieck ring
of the category of finite (left) $G$-sets.
We denote the image of $X$ in $B(G)$ by $[X]$,
and so the operations in $B(G)$ are induced by
$[X]+[Y]=[X\sqcup Y]$ and $[X][Y]=[X\times Y]$.
Now $B(G)$ is a free abelian group
with a $\Z$-basis given by $[G/H]$,
as $H$ runs over a set of representatives of conjugacy classes of subgroups of $G$.
And in particular, $[G/G]$ is the multiplicative identity.

Now for any $G$-set $X$ and subgroup $H\leq G$,
we let $X^H$ denote the $H$-fixed points of $X$.
We then have a ring morphism $\phi_G:B(G)\longrightarrow\prod_{H\leq G}\Z$
such that $\phi_G([X])=(|X^H|)_{H\leq G}$ for any $G$-set $X$.
We see that $G$ acts on the codomain of this map by conjugation,
but the image of the map lies in the fixed points of this action since $X^{(\prescript{g}{}{H})}=gX^H$.
So we set $\tilde B(G):=\left(\prod_{H\leq G}\Z\right)^G$, which we call the ghost ring of $B(G)$.
And then we have the ghost map
$$\phi_G:B(G)\longrightarrow\tilde B(G),\quad a\mapsto (|a^H|)_{H\leq G},$$
which is well-known to be an injective ring morphism with finite cokernel.
We can see then that the unit group $B(G)^\times$ is isomorphic
to a subgroup of $\tilde B(G)^\times=\left(\prod_{H\leq G}\{\pm1\}\right)^G.$
Hence $B(G)^\times$ is simply a finite elementary abelian 2-group,
and a common problem is computing its rank as such.
One well-known result due to Dress is that when $G$ has odd order, $B(G)^\times=\{\pm1\}$.
This will be important later.

When $G$ and $H$ are two finite groups,
we define a $(G,H)$-biset to be a left $G$-set $U$ that is also a right $H$-set
such that $g(uh)=(gu)h$ for all $g\in G, u\in U, h\in H$.
If additionally $X$ is a left $H$-set, then $U\times X$ is a left $H$-set via $h(u,x)=(uh^{-1},hx)$.
We then define the tensor product $U\times_HX$
to be the set of $H$-orbits of $U\times X$ under this action.
The $H$-orbit of $(u,x)$ is denoted by $(u,_Hx)$,
and $G$ acts on $U\times_HX$ via $g(u,_Hx)=(gu,_Hx)$.
Then $U$ defines an additive function
$B(U):B(H)\longrightarrow B(G),[X]\mapsto[U\times_HX]$.
And there is a unique additive function $\tilde B(U):\tilde B(H)\longrightarrow\tilde B(G)$
such that $\phi_G\circ B(U)=\tilde B(U)\circ\phi_H$.
We describe $\tilde B(U)$ in more detail for particular bisets.
The proofs of these claims here can be found in \cite{C}, for instance.

When $H\leq G$,
we define the restriction $(H,G)$-biset $\Res^G_H=G$
with left-action given by multiplication by elements of $H$
and right-action given by multiplication by elements of $G$.
And we similarly define the induction $(G,H)-$biset $\Ind^G_H=G$
with multiplication from $G$ on the left and $H$ on the right.

When $N\unlhd G$,
we define the inflation $(G,G/N)$-biset $\Inf^G_{G/N}=G/N$
with right-action given by multiplication by elements of $G/N$
and left-action given by projecting elements of $G$ onto $G/N$ and the multiplying in $G/N$.
Similarly, we define the deflation $(G/N,G)$-biset $\Def^G_{G/N}=G/N$.

And when $\gamma:G\longrightarrow G'$ is an isomorphism of groups,
we define the transport by isomorphism $(G',G)$-biset $\Iso(\alpha)$ to be $G'$ with actions given by
$x\cdot y\cdot g=xy\gamma(g)$.
For these 5 types of elementary bisets,
we have their maps between ghost rings given by the following:
\begin{align*}
\tilde B(\Res^G_H)&:\tilde B(G)\longrightarrow\tilde B(H), &
(n_S)_{S\leq G}&\mapsto(n_T)_{T\leq H}\\
\tilde B(\Ind^G_H)&:\tilde B(H)\longrightarrow\tilde B(G), &
(n_T)_{T\leq H}&\mapsto\left(\sum_{g\in S\backslash G/H}n_{S^g\cap H}\right)_{S\leq G}\\
\tilde B(\Inf^G_{G/N})&:\tilde B(G/N)\longrightarrow\tilde B(G), &
(n_{X/N})_{X/N\leq G/N}&\mapsto(n_{SN/N})_{S\leq G}\\
\tilde B(\Def^G_{G/N})&:B(G)\longrightarrow\tilde B(G/N), &
(n_S)_{S\leq G}&\mapsto(n_X)_{X/N\leq G/N}\\
\tilde B(\Iso(\gamma))&:\tilde B(G)\longrightarrow\tilde B(G'), &
(n_S)_{S\leq G}&\mapsto(n_{\alpha^{-1}(S')})_{S'\leq G'}
\end{align*}
These operations are used to form what is called a biset functor,
but since we will not use this idea, we will not discuss further.
For any $H\leq G$,
we set $\mathrm{Indinf}^G_{N_G(H)/H}=\Ind^G_{N_G(H)}\times_{N_G(H)}\Inf^{N_G(H)}_{N_G(H)/H}$
and $\mathrm{Defres}^G_{N_G(H)/H}=\Def^{N_G(H)}_{N_G(H)/H}\times_{N_G(H)}\Res^G_{N_G(H)}$.

Also every $(G,H)$-biset $U$ defines a multiplicative map $B(H)\longrightarrow B(G)$
that maps the class of an $H$-set $X$ to the class of $\Hom_H(U^{op},X),$
the set of all function $f:U\longrightarrow X$ such that $f(uh)=h^{-1}f(u)$ for all $u\in U$ and $h\in H$.
This is a $G$-set via $(gf)(u)=f(g^{-1}u).$
So we have a group homomorphism $B(U)^\times:B(H)^\times\longrightarrow B(G)^\times$.
When $U$ is a restriction, inflation, deflation, or transport by isomorphism biset,
then $B(U)$ is a ring morphism and coincides with $B(U)^\times$.
This is not true however, for $\Ind^G_H$,
as $B(\Ind^G_H)$ is not multiplicative.
Slightly abusing notation, we denote $B(\Ind^G_H)^\times$ by $B(\mathrm{Ten}^G_H)$,
and refer to this group homomorphism as tensor induction.
Then $\tilde B(\mathrm{Ten}^G_H)$ is defined just as $\tilde B(\Ind^G_H)$,
except the sum indexed by elements of $G$ is replaced with a product.

\section{Background on Fusion Systems}\label{bgFusion}
We describe in this section the basics of fusion systems without too much explanation,
attempting to cover only what we will use later.
For more precise definitions/explanations, see \cite{AKO}, for instance.
Let $p$ be a prime number and $S$ a finite $p$-group.
A fusion system on $S$ is a category $\mathcal F$
whose objects are the subgroups of $S$
and whose morphisms are injective group homomorphisms such that for $P,Q\leq S,$
one has $\mathrm{Hom}_S(P,Q)\subseteq\mathrm{Hom}_{\mathcal F}(P,Q)$
and every $\varphi\in\mathrm{Hom}_{\mathcal F}(P,Q)$
is the composition of an isomorphism in $\mathcal F$ and an inclusion of subgroups of $S$.
For any $P\leq S$,
we set $\Aut_{\mathcal F}(P)=\Hom_{\mathcal F}(P,P)$.
This contains $\Inn(P)$, and so we can define $\Out_{\mathcal F}(P)=\Aut_{\mathcal F}(P)/\Inn(P)$.

Fusion systems are often too general to work with,
and so we work only with fusion systems that are said to be saturated.
A saturated fusion system satisfies a few additional axioms
that resemble the Sylow theorems.
We omit the precise assumptions here for brevity,
but will point out where they are necessary later.
The main source of fusion systems
comes from the case where $G$ is a finite group and $S\in\mathrm{Syl}_p(G)$.
We define the fusion system $\mathcal F_S(G)$ such that
$\Hom_{\mathcal F_S(G)}(P,Q)=\Hom_G(P,Q)$ for all $P,Q\leq S$.
We refer to a fusion system of this form as a Frobenius fusion system.
Such fusion systems are always saturated.
Saturated fusion systems that are not of this form are called exotic fusion systems.

If there exists an isomorphism in $\mathrm{Hom}_{\mathcal F}(P,Q)$,
then we say that $P$ and $Q$ are $\mathcal F$-conjugate,
and write $P\cong_{\mathcal F}Q$.
We denote the set of all such isomorphisms by $\mathrm{Iso}_{\mathcal F}(P,Q)$,
and let $P^{\mathcal F}$ denote the set of all $\mathcal F$-conjugates of $P$.
Similarly, if $x,y\in S$,
then we say they are $\mathcal F$-conjugate
if there exists a $\varphi\in\mathrm{Iso}_{\mathcal F}(\langle x\rangle,\langle y\rangle)$
such that $\varphi(x)=y$.
We let $x^{\mathcal F}$ denote the $\mathcal F$-conjugacy class of $x$.
A subgroup $P\leq S$ is said to be strongly closed in $\mathcal F$
if $x^{\mathcal F}\subseteq P$ for all $x\in P$.
And $P$ is said to be normal in $\mathcal F$, denoted by $P\unlhd\mathcal F$,
if for any $\varphi\in\mathrm{Hom}_{\mathcal F}(Q,R)$,
there exists a $\psi\in\mathrm{Hom}_{\mathcal F}(QP,RP)$
such that $\restr{\psi}{Q}=\varphi$ and $\psi(P)=P$.
Since $S$ is a $p$-group,
then so is $\Aut_S(P)$ for any $P\leq S$.
We call $P$ fully $\mathcal F$-automized if $\Aut_S(P)\in\mathrm{Syl}_p(\Aut_{\mathcal F}(P))$.

\section{Burnside Ring of a Fusion System}\label{BFusion}

Let $p$ be a prime number and fix an $S\in\mathrm{Syl}_p(G)$.
Let $\mathcal F_S(G)$ denote the Frobenius fusion system on $S$ induced by $G$.
Following \cite{GRY}, we define the Burnside ring of $\mathcal F_S(G)$
to be the subring $B(\mathcal F_S(G))$ of $B(S)$
consisting of the elements $a\in B(S)$ such that $|a^P|=|a^Q|$ whenever $P,Q$
are subgroups of $S$ that are conjugate in $G$.
More generally, for any saturated fusion system $\mathcal F$ on $S$,
we define the Burnside ring of $\mathcal F$
to be the subring $B(\mathcal F)$ of $B(S)$
consisting of the elements $a\in B(S)$ such that $|a^P|=|a^Q|$
whenever $P$ and $Q$ are isomorphic in $\mathcal F$.
The first main result about $B(\mathcal F)$ is due to Sune Reeh.

\begin{theorem}[\cite{Reeh}, Thm. A]
Let $\mathcal F$ be a saturated fusion system on a $p$-group $S$.
Then $B(\mathcal F)$ is free abelian
with $\Z$-rank equal to the number of $\mathcal F$-conjugacy classes
of subgroups of $S$.
\end{theorem}

This theorem also gives a canonical  basis of $B(\mathcal F)$
indexed by the $\mathcal F$-conjugacy classes of $S$,
but we use this result only in the next example, so we omit a general explanation.
We wish to explore what information Burnside rings of various fusions systems
can and cannot recover about the fusion systems and the underlying $p$-groups.
First we show that $B(\mathcal F)$ does not uniquely determine $S$ up to isomorphism.
This is not too surprising, however, since the Burnside ring of a group
cannot always determine the group's isomorphism type.
See \cite{GRLV}, for instance.

\begin{example}
Let $S=V_4$ be the Klein 4-group,
$T=C_4$ be the cyclic group of order 4,
and $G=A_4$, the alternating group of order 12.
Then $S\in\mathrm{Syl}_p(G)$,
and we can define the saturated fusion system $\mathcal F:=\mathcal F_S(G)$.
Now $S$ has $3$ distinct maximal subgroups of index $2$, but they are all conjugate in $G$.
Hence $\mathcal F$ has only 3 conjugacy classes,
and thus $B(\mathcal F)$ has a $\Z$-basis: $\alpha_1,\alpha_M,\alpha_S$,
where $M$ represents any of the maximal subgroups of $S$.
Also let $\mathcal F':=\mathcal F_T(T)$ be the trivial fusion system on $T$.
Then $B(\mathcal F')$ has a $\Z$-basis given by $\beta_1,\beta_N,\beta_T$,
where $N$ represents the unique maximal subgroup of $T$.
The tables of marks for the two fusion systems are given below:
\begin{center}
\begin{tabular}{c|ccc}
$\mathcal F$&1&M&S\\
\hline
$\alpha_1$&$4$&0&0\\
$\alpha_M$&$6$&2&0\\
$\alpha_S$&1&1&1
\end{tabular}
\hspace{1.5cm}
\begin{tabular}{c|ccc}
$\mathcal F'$&1&N&T\\
\hline
$\beta_1$&$4$&0&0\\
$\beta_N$&$2$&2&0\\
$\beta_T$&1&1&1
\end{tabular}
\end{center}
We define a map $B(\mathcal F)\longrightarrow B(\mathcal F')$
by $\alpha_1\mapsto\beta_1, \alpha_M\mapsto\beta_N+\beta_1,\alpha_S\mapsto\beta_T$,
and extend linearly.
Thus we have an isomorphism of free abelian groups of rank 3,
and we leave it to the reader to check that this is an isomorphism of rings.
Hence $B(\mathcal F_{V_4}(A_4)\cong B(\mathcal F_{C_4}(C_4))$.
\end{example}

So we only now consider the case with a single fixed $p$-group $S$
and two fusion systems $\mathcal F$ and $\mathcal F'$ on $S$.
If $\mathcal F\subseteq\mathcal F',$ then $B(\mathcal F')\subseteq B(\mathcal F)$.
But if $\mathcal F\subsetneq\mathcal F',$ this does not imply that
$B(\mathcal F')\subsetneq B(\mathcal F)$.

\begin{example}
Let $p$ be an odd prime, $S=C_p$, and $\mathcal F=\mathcal F_S(S)$ be the trivial fusion system on $S$.
Now $S$ is a Sylow 2-subgroup of $D_{2p}$, the dihedral group of order $2p$.
So let $\mathcal F'=\mathcal F_S(D_{2p})$.
Then $B(\mathcal F')=B(\mathcal F)=B(S)$
since these are all free abelian of rank 2, but $\mathcal F\subsetneq\mathcal F'$.
Notice that $\mathrm{Aut}_{\mathcal F}(S)=\mathrm{Inn}(S)$, which is trivial,
but $\mathrm{Aut}_{\mathcal F'}(S)=\mathrm{Aut}_{D_{2p}}(S)\cong C_2$,
as the elements of order 2 in $D_{2p}$ invert a generator of $S$ via conjugation.
\end{example}

More generally for an odd prime $p$, if $A\leq\Aut(C_p)\cong C_{p-1}$,
then $C_p\in\mathrm{Syl}_p(C_p\rtimes A)$,
and we have the fusion system $\mathcal F_{C_p}(C_p\rtimes A)$,
but its Burnside ring is exactly $B(C_p)$ no matter what $A$ is.
Notice this example shines no light on the $p=2$ case,
but we also present the following example.

\begin{example}
Let $S$ be the elementary abelian 2-group of rank 3.
Then $A:=\mathrm{Aut}(S)\cong GL(3,2).$
Let $T\in\mathrm{Syl}_7(A)$, and set $N:=N_A(T)$.
Then $T\cong C_7$ and $N\cong C_7\rtimes C_3$.
Notice $T$ transitively permutes the 7 involutions of $S$,
hence so does $N$.
Set $G:=S\rtimes N$ and $H:=S\rtimes T$.
Then in $H$, all the involutions of $S$ are fused,
hence also all the subgroups of $S$ of order 4 are fused.
The same is true for $G$.
Thus if we set $\mathcal F_1:=\mathcal F_S(H)$ and $\mathcal F_2:=\mathcal F_S(G)$,
we have $\mathcal F_1\subseteq\mathcal F_2$,
and there are 4 conjugacy classes of subgroups of $S$ (one per possible $2$-rank of subgroups)
in both $\mathcal F_1$ and $\mathcal F_2$.
Thus $B(\mathcal F_1)=B(\mathcal F_2)\cong\Z^4,$ as free abelian groups.
But we see that $\mathcal F_1\neq\mathcal F_2$
since $\mathrm{Aut}_{\mathcal F_1}(S)=\mathrm{Aut}_H(S)\cong N_H(S)/C_H(S)=(S\rtimes T)/S\cong T$
and $\mathrm{Aut}_{\mathcal F_2}(S)\cong N.$
\end{example}

The point here is that $B(\mathcal F)$ does not see the automorphism groups
$\mathrm{Aut}_{\mathcal F}(P)$ for $P\leq S$.
And these automorphism groups determine the fusion system as $P$
ranges over $S$ plus the essential subgroups of $\mathcal F$, due to Alperin's Fusion Theorem.
So, in general, $\mathcal F$ cannot be recovered from $B(\mathcal F)$ alone.
We are tempted to speculate that properly contained fusion systems on 2-groups
lead to properly contained unit groups for their Burnside rings,
but the following example shows that that is also not true.

\begin{example}
Let $S:=C_2^4$ be the elementary abelian 2-group of rank 4.
Then $\mathrm{Aut}(S)=GL(4,2)$ has subgroups $F\leq E$ with $F=C_7$ and $E=C_7\rtimes C_3$.
Set $\mathcal F_1:=\mathcal F_S(S\rtimes F)$ and $\mathcal F_2:=\mathcal F_S(S\rtimes E)$.
Then $\mathcal F_1\subseteq \mathcal F_2$.
There are 13 $\mathcal F_1$-classes, 3 of which contain maximal subgroups.
And there 11 $\mathcal F_2$-classes, 3 of which contain maximal subgroups.
Hence $B(\mathcal F_2)$ is a proper subring of $B(\mathcal F_1)$,
but $B(\mathcal F_1)^\times=B(\mathcal F_2)^\times$ since both
are elementary abelian 2-groups of rank 4.
\end{example}

\section{Outer Automorphism Action}\label{OutAction}
Let $\mathcal F$ be a saturated fusion system on a $p$-group $S$.
Then for every $P\leq S$, we have the automorphism group $\mathrm{Aut}(P)$,
which acts on $B(P)$ via ring morphisms in the following way:
If $\alpha\in\mathrm{Aut}(P)$, then we have the elementary bifree $(P,P)$-biset $\mathrm{Iso}(\alpha)$,
which defines a ring automorphism $B(\mathrm{Iso}(\alpha)):B(P)\longrightarrow B(P)$
such that $[X]\mapsto[\mathrm{Iso}(\alpha)\times_PX]$.
Since $\mathrm{Iso}(\alpha\circ\beta)\cong\mathrm{Iso}(\alpha)\times_P\mathrm{Iso}(\beta)$
for all $\alpha,\beta\in\mathrm{Aut}(P)$,
this defines an action of $\mathrm{Aut}(P)$ on $B(P)$ via ring automorphisms.
We denote the action simply by $\alpha\cdot b$ for all $\alpha\in\mathrm{Aut}(P)$ and $b\in B(P)$.
Notice that $\mathrm{Inn}(P)$ acts trivially,
So we can consider the action of $\mathrm{Out}(P)$ on $B(P)$.
We denote the image of $\alpha$ in $\mathrm{Out}(P)$ by $\bar\alpha$.

We also see that $\mathrm{Aut}(P)$ acts via ring automorphisms on the ghost ring $\tilde B(P)$ by
$$\tilde B(\mathrm{Iso}(\alpha)):\tilde B(P)\longrightarrow\tilde B(P),\quad
(b_Q)_{Q\leq P}\mapsto(b_{\alpha^{-1}(Q)})_{Q\leq P},$$
and we can see that $\phi_P(\alpha\cdot b)=\alpha\cdot\phi_P(b)$ for all $b\in B(P)$
(see \cite{C}, for example).
That is, the ghost map is $\mathrm{Aut}(P)$-equivariant.
But $\mathrm{Inn}(P)$ also acts trivially on $\tilde B(P)$,
and so we consider the $\mathrm{Out}(P)$-action on $\tilde B(P)$,
and note again that the ghost map is $\mathrm{Out}(P)$-equivariant.

For any subgroup $\Gamma\leq\mathrm{Out}(P)$,
we can consider the fixed point subrings
$B(P)^\Gamma\subseteq B(P)$ and $\tilde B(P)^\Gamma\subseteq\tilde B(P)$.
Since the ghost map is $\mathrm{Out}(P)$-equivariant and injective,
then for any $b\in B(P)$, we have $b\in B(P)^\Gamma$ iff $\phi_P(b)\in\tilde B(P)^\Gamma$.
We will mostly be interested in the subgroups
$\mathrm{Out}_S(P)\leq\mathrm{Out}_{\mathcal F}(P)\leq\mathrm{Out}(P)$,
which correspond to the subrings
$B(P)^{\mathrm{Out}_{\mathcal F}(P)}\subseteq B(P)^{\mathrm{Out}_S(P)}.$

Now for any $P\leq S$, we see that
the image of $B(\mathrm{Res}_P^S):B(S)\longrightarrow B(P)$
is contained inside $B(P)^{\mathrm{Out}_S(P)}$,
and we can also see that
$B(\mathrm{Res}_P^S)(B(\mathcal F))\subseteq B(P)^{\mathrm{Out}_{\mathcal F}(P)}$.
We show that this also gives sufficient conditions for an element of $B(S)$ to be in $B(\mathcal F)$.
Moreover, we only need to check for $S$ and $\mathcal F$-essential subgroups of $S$.

\begin{theorem}\label{stablecheck}
Let $\mathcal F$ be a saturated fusion system on a $p$-group $S$, and let $b\in B(S)$.
\begin{enumerate}
\item[(a.)] $b\in B(\mathcal F)$ iff $B(\mathrm{Res}^S_P)(b)\in B(P)^{\mathrm{Out}_{\mathcal F}(P)}$
for all $P\leq S$.
\item[(b.)] $b\in B(\mathcal F)$ iff $B(\mathrm{Res}^S_R)(b)\in B(R)^{\mathrm{Out}_{\mathcal F}(R)}$
for $R=S$ and every $\mathcal F$-essential $R<S$.
\end{enumerate}
\end{theorem}
\begin{proof}
We have already seen that the image under restriction of $B(\mathcal F)$
is in the $\mathrm{Out}_{\mathcal F}(P)$-fixed points of $B(P)$ for every $P\leq S$.
So we need only prove the second statement, and this will imply the first.
Assume that $B(\mathrm{Res}^S_R)(b)\in B(R)^{\mathrm{Out}_{\mathcal F}(R)}$
for $R=S$ and all $\mathcal F$-essential $R<S$.
And suppose that $\phi_S(b)=(|b^P|)_{P\leq S}$.
Let $P,Q\leq S$ such that $P$ and $Q$ are $\mathcal F$-conjugate.
We wish to show that $|b^P|=|b^Q|$.
Let us pick a $\varphi\in\mathrm{Iso}_{\mathcal F}(P,Q)$.
Then by Alperin's Fusion Theorem (see \cite{AKO}, Theorem 3.6),
there exists subgroups $P=P_0,P_1,\dots,P_k=Q$,
$R_i\geq\langle P_{i-1},P_i\rangle$ with $R_i=S$ or $R_i$ being $\mathcal F$-essential,
and $\varphi_i\in\mathrm{Aut}_{\mathcal F}(R_i)$ such that
$\varphi_i(P_{i-1})\leq P_i$ and $\varphi=(\varphi_k\mid_{P_{k-1}})\circ\cdots\circ(\varphi_1\mid_{P_0})$.
Since $\varphi$ is an isomorphism, the $P_i$ must all have the same order.
Hence $\varphi_i(P_{i-1})=P_i$ for $i=1,\dots,k$.

By the assumption, $B(\mathrm{Res}^S_{R_i})(b)\in B(R_i)^{\mathrm{Out}_{\mathcal F}(R_i)}$ for each $i$.
So $$(|b^U|)_{U\leq R_i}=\tilde B(\mathrm{Res}^S_{R_i})(\phi_S(b))
=\phi_{R_i}(B(\mathrm{Res}^S_{R_i})(b))
\in\phi_{R_i}(B(R_i)^{\mathrm{Out}_{\mathcal F}(R_i)})=\tilde B(R_i)^{\mathrm{Out}_{\mathcal F}(R_i)}.$$
Then since $\varphi_i\in\mathrm{Aut}_{\mathcal F}(R_i)$,
we have $\bar\varphi_i\in\mathrm{Out}_{\mathcal F}(R_i)$, and
$$(|b^U|)_{U\leq R_i}=\bar\varphi_i\cdot(|b^U|)_{U\leq R_i}
=(|b^{\varphi_i^{-1}(U)}|)_{U\leq R_i}.$$
In particular, $|b^{P_i}|=|b^{\varphi_i^{-1}(P_i)}|=|b^{P_{i-1}}|$ for $i=1,\dots,k$.
Putting these altogether, we have $|b^Q|=|b^P|$.
This shows that $b\in B(\mathcal F)$.
\end{proof}

\section{Units of $B(\mathcal F)$}\label{Units}

We now focus on units of $B(\mathcal F)$.
For any $p$-group $S$, the unit group $B(S)^\times$ is an elementary abelian 2-group,
e.g. an $\mathbb F_2$-vector space.
Thus $B(\mathcal F)^\times\leq B(S)^\times$ is also an $\mathbb F_2$-vector space.
We wish to compute $\mathrm{dim}_{\mathbb F_2}(B(\mathcal F)^\times)$.
When $p>2,$ then $S$ is an odd order group, and $B(S)^\times=\{\pm1\}$.
Then also $B(\mathcal F)^\times={\pm1}$ for any fusion system $\mathcal F$ on $S$.
So we are mostly interested here in the case when $p=2$,
though we will make statements for an arbitrary $p$ when applicable.

Since for $P\leq S$, we have $\mathrm{Out}(P)$ acting on $B(P)$ via ring automorphisms,
we have $\mathrm{Out}(P)$ acting on $B(P)^\times$ via group automorphisms.
Therefore $B(P)^\times$ is an $\mathbb F_2\mathrm{Out}(P)$-module.
For any $\Gamma\leq\mathrm{Out}(P)$,
we have the fixed points $(B(P)^\times)^\Gamma$,
which is then an $\mathbb F_2$-subspace of $B(P)^\times$.
And if $\Delta\leq\Gamma,$ then $(B(P)^\times)^\Gamma\leq(B(P)^\times)^\Delta$.
In particular, on the unit groups $B(\mathrm{Res}^S_P):B(S)^\times\longrightarrow B(P)^\times$
has its image inside $(B(P)^\times)^{\mathrm{Out}_S(P)}$, and
$B(\mathrm{Res}^S_P)(B(\mathcal F)^\times)\leq(B(P)^\times)^{\mathrm{Out}_{\mathcal F}(P)}$.
So restriction induces an $\mathbb F_2$-linear map
$$d^{\mathcal F}_P:B(S)^\times\longrightarrow
(B(P)^\times)^{\mathrm{Out}_S(P)}/(B(P)^\times)^{\mathrm{Out}_{\mathcal F}(P)},$$
with $B(\mathcal F)^\times\leq\mathrm{Ker}(d^{\mathcal F}_P)$ for every $P\leq S$.
We can similarly define $\tilde d^{\mathcal F}_P$ on the ghost ring level.
Now Theorem \ref{stablecheck} implies the following.

\begin{theorem}\label{kernels}
Let $\mathcal F$ be a saturated fusion system on a $p$-group $S$.
\begin{enumerate}
\item $B(\mathcal F)^\times=\displaystyle\bigcap_{P\leq S}\mathrm{Ker}(d^{\mathcal F}_P).$
\item $B(\mathcal F)^\times=\displaystyle\bigcap_R\mathrm{Ker}(d^{\mathcal F}_R),$
where $R$ runs over the subgroups of $S$ such that $R=S$ or $R$ is $\mathcal F$-essential.
\end{enumerate}
\end{theorem}

Notice for $P=S$,
we have simply $\mathrm{Ker}(d^{\mathcal F}_S)=(B(S)^\times)^{\mathrm{Out}_{\mathcal F}(S)}$.
So in particular, when $\mathcal F$ has no proper essential subgroups,
$B(\mathcal F)^\times$ is completely determined by the action of $\mathrm{Out}_{\mathcal F}(S)$
on $B(S)^\times.$

\begin{corollary}
Let $\mathcal F$ be a saturated fusion system on a $p$-group $S$
such that $S$  has no proper $\mathcal F$-essential subgroups.
Then $B(\mathcal F)^\times=(B(S)^\times)^{\mathrm{Out}_{\mathcal F}(S)}$.
\end{corollary}

Now if $P\leq S$ and $\Delta\leq \Gamma\leq\mathrm{Out}(P)$,
then we have $(B(P)^\times)^\Gamma\leq(B(P)^\times)^\Delta$ as $\mathbb F_2$-vector spaces,
and we have the familiar trace map
$$\mathrm{tr}_\Delta^\Gamma:(B(P)^\times)^\Delta\longrightarrow(B(P)^\times)^\Gamma,
\quad b\mapsto\prod_{\alpha\in\Gamma/\Delta}(\alpha\cdot b),$$
where $\alpha$ runs over any choice of coset representatives of $\Gamma/\Delta$.
This map is independent of the choice of representatives and is $\mathbb F_2$-linear,
i.e., a group homomorphism.
And since the ghost map is $\mathrm{Out}(P)$-equivariant,
we also have the extension to the unit groups of the ghost rings:
$\widetilde{\mathrm{tr}}^\Gamma_\Delta:
(\tilde B(P)^\times)^\Delta\longrightarrow(\tilde B(P)^\times)^\Gamma$.
We are most interested in the case where
$\Delta=\mathrm{Out}_S(P)$ and $\Gamma=\mathrm{Out}_{\mathcal F}(P)$.
In this case, we denote the relevant trace maps by simply
$\mathrm{tr}^{\mathcal F}_P$ and $\widetilde{\mathrm{tr}}^{\mathcal F}_P$
so that we have
$$\mathrm{tr}^{\mathcal F}_P:
(B(P)^\times)^{\mathrm{Out}_S(P)}\longrightarrow(B(P)^\times)^{\mathrm{Out}_{\mathcal F}(P)},
\quad b\mapsto\prod_{\alpha\in\mathrm{Out}_{\mathcal F}(P)/\mathrm{Out}_S(P)}(\alpha\cdot b)$$
and similarly for $\widetilde{\mathrm{tr}}^{\mathcal F}_P$.
Recall that the restriction map $B(\mathrm{Res}_P^S):B(S)^\times\longrightarrow B(P)^\times$
always has its image inside of $(B(P)^\times)^{\mathrm{Out}_S(P)}$,
so we can consider the composition
$$\mathrm{tr}_P^{\mathcal F}\circ B(\mathrm{Res}_P^S):
B(S)^\times\longrightarrow(B(P)^\times)^{\mathrm{Out}_{\mathcal F}(P)}.$$
Similarly, we have
$\widetilde{\mathrm{tr}}^{\mathcal F}_P\circ\tilde B(\mathrm{Res}^S_P):
\tilde B(S)^\times\longrightarrow(\tilde B(P)^\times)^{\mathrm{Out}_{\mathcal F}(P)},$
which commutes with the ghost maps.

\begin{proposition}\label{splitsurj}
Let $\mathcal F$ be a saturated fusion system on a 2-group $S$,
and suppose that $P\leq S$ is fully $\mathcal F$-automized.
Then $\mathrm{tr}^{\mathcal F}_P$ is split surjective, and
$$(B(P)^\times)^{\mathrm{Out}_S(P)}
=\mathrm{Ker}(\mathrm{tr}^{\mathcal F}_P)\times(B(P)^\times)^{\mathrm{Out}_{\mathcal F}(P)}.$$
\end{proposition}
\begin{proof}
Since $P$ is fully $\mathcal F$-automized, we have
$\mathrm{Aut}_S(P)\in\mathrm{Syl}_2(\mathrm{Aut}_{\mathcal F}(P))$.
Hence $\mathrm{Out}_S(P)\in\mathrm{Syl}_2(\mathrm{Out}_{\mathcal F}(P))$,
and $2\nmid[\mathrm{Out}_{\mathcal F}(P):\mathrm{Out}_S(P)]$.
Thus if $b\in(B(P)^\times)^{\mathrm{Out}_{\mathcal F}(P)}$,
then $\mathrm{tr}_P^{\mathcal F}(b)=b^{[\mathrm{Out}_{\mathcal F}(P):\mathrm{Out}_S(P)]}=b$.
So in this case, $\mathrm{tr}_P^{\mathcal F}$ is split surjective, and
$(B(P)^\times)^{\mathrm{Out}_S(P)}
=\mathrm{Ker}(\mathrm{tr}^{\mathcal F}_P)\times(B(P)^\times)^{\mathrm{Out}_{\mathcal F}(P)}.$
\end{proof}

Since $S$ itself is always fully automized in any saturated fusion system, we always have that
$\mathrm{tr}^{\mathcal F}_S:B(S)^\times\longrightarrow(B(S)^\times)^{\mathrm{Out}_{\mathcal F}(S)}$
is split surjective with
$$B(S)^\times
=\mathrm{Ker}(\mathrm{tr}^{\mathcal F}_S)\times(B(S)^\times)^{\mathrm{Out}_{\mathcal F}(S)},$$
and we notice that $B(\mathcal F)^\times\leq(B(S)^\times)^{\mathrm{Out}_{\mathcal F}(S)}$.

\section{The Frobenius Case}\label{Frobenius}
In this section, we consider the case of the Frobenius fusion system $\mathcal F=\mathcal F_S(G)$,
where $G$ is a finite group and $S\in\mathrm{Syl}_p(G)$.
We see that the restriction map
$B(\mathrm{Res}^G_S):B(G)\longrightarrow B(S)$ has its image inside $B(\mathcal F)$.
We first show that $B(\mathcal F)$ is exactly the image of $B(\mathrm{Res}^G_S)$.
This result was communicated to us by Robert Boltje,
based on the idea of the $\ast$-construction of characters
from Lluis Puig and Michel Brou\'{e} detailed in \cite{BP}.

\begin{theorem}\label{Boltje}
Let $G$ be a finite group, $S\in\mathrm{Syl}_p(G)$ for some prime $p$,
and $\mathcal F=\mathcal F_S(G)$, the Frobenius fusion system on $S$ from $G$.
Then the image of $B(\mathrm{Res}^G_S):B(G)\longrightarrow B(S)$ is $B(\mathcal F)$.
\end{theorem}
\begin{proof}
To show this, we will create an inective ring morphism $t^G_S:B(\mathcal F)\longrightarrow B(G)$
such that $B(\mathrm{Res}^G_S)\circ t^G_S$ is the identity.
First, for any $a\in B(G)$ and $b\in B(\mathcal F)$,
we create an element $a\ast b\in\tilde B(G)$ in the following way:
For any $H\leq G,$ we set $|(a\ast b)^H|=|a^H||b^P|$,
where $P$ is any Sylow $p$-subgroup of $H$ contained in $S$.
Since $b\in B(\mathcal F),$ this does not depend on a choice of $P$.

We wish to show that in fact $a\ast b\in\phi_G(B(G))$.
By \cite{Boltje}, Corollary 2.5, it suffices to show
for any $H\leq G,$ any prime $q$, and any $Q/H\in\mathrm{Syl}_q(N_G(H)/H)$,
that the congruence
$$\sum_{H\leq I\leq Q}\mu(H,I)|(a\ast b)^I|\equiv0\pmod{[Q:H]}$$
holds, where $\mu$ is the M\"{o}bius function on the poset of subgroups of $G$.
We distinguish two cases:

When $p\neq q$ and $Q/H\in\mathrm{Syl}_q(N_G(H)/H),$
we choose a fixed $P\in\mathrm{Syl}_p(H)$ that is contained in $S$.
Then for $H\leq I\leq Q,$ also $P$ is a Sylow $p$-subgroup of $I$.
Hence we have
$$\sum_{H\leq I\leq Q}\mu(H,I)|(a\ast b)^I|
=\sum_{H\leq I\leq Q}\mu(H,I)|a^I||b^P|
=|b^P|\sum_{H\leq I\leq Q}\mu(H,I)|a^I|.$$
But since $a\in B(G),$ then $\sum_{H\leq I\leq Q}\mu(H,I)|a^I|\equiv0\pmod{[Q:H]},$
and thus multiplying by $|b^P|,$ the original sum is still congruent to 0.

When $p=q$ and $Q/H\in\mathrm{Syl}_p(N_G(H)/H),$
we pick a $P\in\mathrm{Syl}_p(Q)$ and a $g\in G$ such that $\prescript{g}{}{P}\leq S$.
If we set $\hat P:=\prescript{g}{}{P}$ and $\hat P_0:=\prescript{g}{}{(P\cap H)},$
then we have $\hat P_0\unlhd\hat P\leq S$
and an isomorphism $\alpha: Q/H\longrightarrow \hat P/\hat P_0$.
Then $\alpha$ induces an isomorphism between the poset intervals
$[H,Q]$ and $[\hat P_0,\hat P]$.
For any subgroup $I$ in the interval $[H,Q],$
let $P_I$ denote the corresponding subgroup in the interval $[\hat P_0,\hat P]$.
Then also $P_I$ is a Sylow $p$-subgroup of $I$ contained in $S$.
Now set $\bar a:=B(\mathrm{Defres}^G_{Q/H})(a)\in B(Q/H)$ and then
$x:=B(\mathrm{Iso}(\alpha))(\bar a)\in B(\hat P/\hat P_0)$.
Next, let $y:=B(\mathrm{Defres}^S_{\hat P/\hat P_0})(b)\in B(\hat P/\hat P_0)$.
Then for every $H\leq I\leq Q$, we have
$$|(a\ast b)^I|
=|a^I||b^{P_I}|
=|\bar a^{I/H}||y^{P_I/\hat P_0}|
=|x^{P_I/\hat P_0}||y^{P_I/\hat P_0}|
=|(xy)^{P_I/\hat P_0}|.$$
And since $xy\in B(\hat P/\hat P_0),$ we have
$$0\equiv\sum_{1\leq P_I/\hat P_0\leq\hat P/\hat P_0}\mu(1,P_I/\hat P_0)|(xy)^{P_I/\hat P_0}|
=\sum_{H\leq I\leq Q}\mu(H,I)|(a\ast b)^I|\pmod{[Q:H]},$$
since $[Q:H]=[\hat P:\hat P_0]$.
This shows that in fact $a\ast b\in\phi_G(B(G))$.

Now (with slightly abusing notation),
for any $a\in B(G),$ and $b\in B(\mathcal F),$
we let $a\ast b$ denote the unique element of $B(G)$ such that
$\phi_G(a\ast b)=(|a^H||b^P|)_{H\leq G}$
where $P$ denotes a Sylow $p$-subgroup of $H$ contained in $S$.
It is easy to see that $a\ast b$ is additive in both $a$ and $b$.
Hence we have a map
$$B(G)\otimes_\Z B(\mathcal F)\longrightarrow B(G),\quad a\otimes b\mapsto a\ast b.$$
Moreover, one has
$B(\mathrm{Res}^G_S)(a\ast b)=B(\mathrm{Res}^G_S)(a)b$ in $B(\mathcal F)$
and $1\ast(bc)=(1\ast b)(1\ast c)$ for $b,c\in B(\mathcal F)$.
In particular, the map
$$t^G_S:B(\mathcal F)\longrightarrow B(G),\quad b\mapsto1\ast b$$
is an injective ring homomorphism such that $B(\mathrm{Res}^G_S)\circ t^G_S$
is the identity on $B(\mathcal F)$.
This shows that $B(\mathrm{Res}^G_S)(B(G))=B(\mathcal F)$ exactly.
\end{proof}

This result is somewhat surprising since we know that not every $\mathcal F_S(G)$-stable $S$-set
comes from the restriction of a $G$-set in general.
In \cite{Reeh}, Example 4.3, Sune Reeh
shows that the transitive $\mathcal F_{D_8}(S_5)$-stable $D_8$-set $D_8/1$
cannot be the restriction of any $S_5$-set.
In light of this previous theorem,
this means that $[D_8/1]=B(\mathrm{Res}^{S_5}_{D_8})(a)$ for some virtual $a\in B(S_5)$,
but $[D_8/1]\neq B(\mathrm{Res}^{S_5}_{D_8})([X])$ for any actual $S_5$-set $X$.

Now in the case where $\mathcal F=\mathcal F_S(G),$
then for any $P\leq S,$ we have
$\mathrm{Out}_{\mathcal F}(P)
=\mathrm{Out}_G(P)
=\mathrm{Aut}_G(P)/\mathrm{Inn}(P)
\cong N_G(P)/(PC_G(P))$.
Clearly the restriction map $B(\mathrm{Res}^G_S)$ has image inside
the $\mathrm{Out}_G(S)$-fixed points.
And in particular when $S$ is a normal Sylow 2-subgroup of $G$,
Serge Bouc proved the following result on the unit groups.

\begin{proposition}[\cite{Bouc}, Prop. 6.5]\label{Boucnorm}
Let $S\in\mathrm{Syl}_2(G)$ with $S\unlhd G$.
Then $B(G)^\times\cong(B(S)^\times)^{\mathrm{Out}_G(S)}$
with mutually inverse isomorphisms given by restriction and tensor induction.
\end{proposition}

Hence if $S$ is a normal Sylow 2-subgroup of $G$,
the previous results give two ways to describe the image of restriction on units:
We have $B(\mathcal F)^\times=(B(S)^\times)^{\mathrm{Out}_G(S)}\cong B(G)^\times$.
Moreover, in this case, $t^G_S=B(\mathrm{Ten}^G_S)$ on $B(\mathcal F)^\times$.

More generally, if $S\in\mathrm{Syl}_2(G)$,
then we have $S\in\mathrm{Syl}_2(N_G(S))$ and $S\unlhd N_G(S)$.
So let us set $N:=N_G(S)$ and $\mathcal F':=\mathcal F_S(N)$.
Then we have $\mathcal F'\subseteq\mathcal F$ and therefore $B(\mathcal F)\subseteq B(\mathcal F')$
and $B(\mathcal F)^\times\leq B(\mathcal F')^\times$.
But from the above, we have
$B(\mathcal F')^\times=(B(S)^\times)^{\mathrm{Out}_N(S)}\cong B(N)^\times$.
Notice that
$$\mathrm{Out}_{\mathcal F'}(S)=\mathrm{Out}_N(S)
=\mathrm{Out}_G(S)=\mathrm{Out}_{\mathcal F}(S),$$
and $B(\mathcal F)^\times\leq B(\mathcal F')^\times=(B(S)^\times)^{\mathrm{Out}_G(S)}$.
(It is not necessarily true
that $\mathrm{Out}_{\mathcal F}(P)=\mathrm{Out}_{\mathcal F'}(P)$
for $P<S$.)
Since $B(\mathrm{Res}^N_S):B(N)^\times\longrightarrow B(\mathcal F')^\times$
is an isomorphism,
$B(\mathrm{Res}^G_S)(B(G)^\times)=B(\mathcal F)^\times,$
and $B(\mathrm{Res}^G_S)=B(\mathrm{Res}^N_S)\circ B(\mathrm{Res}^G_N),$
we see that $B(\mathrm{Res}^G_N)(B(G)^\times)\cong B(\mathcal F)^\times$.
So in general, we have the following commutative diagram,
where the unlabeled arrows are all restriction maps:
\begin{center}
\begin{equation}\label{jamisonspic}
\begin{tikzcd}
& B(G)^\times\ar[dl, two heads] \ar[dr, hookleftarrow, "t^G_S"]&\\
B(\mathrm{Res}^G_N)(B(G)^\times)\ar[rr,"\sim"]
\ar[d, phantom, "\rotatebox{-90}{$\leq$}"]&&
B(\mathcal F_S(G))^\times\ar[d, phantom, "\rotatebox{-90}{$\leq$}"]\\
B(N)^\times\ar[rr,"\sim"]&&B(\mathcal F_S(N))^\times
\end{tikzcd}
\end{equation}
\end{center}
We then collect a few results that can be read from Diagram \ref{jamisonspic}.

\begin{proposition}\label{normalizercon}
Suppose $G$ is a finite group and $S\in\mathrm{Syl}_2(G)$.
Set $N:=N_G(S), \mathcal F=\mathcal F_S(G)$ and $\mathcal F'=\mathcal F_S(N)$.
The following are equivalent:
\begin{enumerate}
\item[(i)] There is an embedding $t^G_N:B(N)^\times\longrightarrow B(G)^\times$
such that $B(\mathrm{Res}^G_N)\circ t^G_N=\mathrm{Id}_{B(N)^\times}$.
\item[(ii)] The map $B(\mathrm{Res}^G_N):B(G)^\times\longrightarrow B(N)^\times$
is split surjective.
\item[(iii)] The image of $B(\mathrm{Res}^N_S):B(N)^\times\longrightarrow B(S)^\times$
is $B(\mathcal F)^\times$.
\item[(iv)] The containment $B(\mathcal F)^\times\leq B(\mathcal F')^\times$
is an equality.
\end{enumerate}
\end{proposition}
\begin{proof}
The fact that (i) and (ii) are equivalent is by definition,
and the fact that (iii) and (iv) are equivalent is obvious.
Suppose that (i) holds.
Then if $u\in B(N)^\times,$ we have
$$B(\mathrm{Res}^N_S)(u)
=B(\mathrm{Res}^N_S)((B(\mathrm{Res}^G_N)\circ t^G_N(u)))
=B(\mathrm{Res}^G_S)(t^G_N(u)),$$
which is in $B(\mathcal F)$.
Conversely, if $b\in B(\mathcal F)^\times$,
then $b=B(\mathrm{Res}^G_S)(a)$ for some $a\in B(G)^\times$ by Theorem \ref{Boltje}.
Then $B(\mathrm{Res}^G_N)(a)\in B(N)^\times$,
and $b=B(\mathrm{Res}^N_S)(B(\mathrm{Res}^G_N)(a))$
is in the image of $B(\mathrm{Res}^N_S)$,
showing (iii) holds.

Lastly, suppose that (iii) holds.
We can then define $t^G_N:B(N)^\times\longrightarrow B(G)^\times$
by $t^G_N:=t^G_S\circ B(\mathrm{Res}^N_S)$.
If $u\in B(N)^\times,$ then $B(\mathrm{Res}_S^N)(u)\in B(\mathcal F)^\times$.
So by Theorem \ref{Boltje}, we have
\begin{align*}
B(\mathrm{Res}^N_S)(u)
&=(B(\mathrm{Res}^G_S)\circ t^G_S\circ B(\mathrm{Res}^N_S))(u)\\
&=(B(\mathrm{Res}_S^G)\circ t^G_N)(u)\\
&=(B(\mathrm{Res}^N_S)\circ B(\mathrm{Res}^G_N)\circ t^G_N)(u).
\end{align*}
This lands in $(B(S)^\times)^{\mathrm{Out}_G(S)}=(B(S)^\times)^{\mathrm{Out}_N(S)},$
and so by \ref{Boucnorm} we can apply $B(\mathrm{Ten}^N_S)$ to both sides and get
$u=(B(\mathrm{Res}^G_N)\circ t^G_N)(u),$
thus showing $B(\mathrm{Res}^G_N)\circ t^G_N$
is the identity on $B(N)^\times,$ and therefore (i) holds.
\end{proof}

This proposition is interesting in that it relates control of fusion to a statement
about restricting the unit group of $B(G)$ to a $p$-local subgroup.
The question is then when do these properties hold?
We say that the ``normalizer controls fusion'' if $\mathcal F'=\mathcal F$.
This would of course imply that $B(\mathcal F)=B(\mathcal F')$
and $B(\mathcal F)^\times=B(\mathcal F')^\times$.
So in this case all 4 conditions in Proposition \ref{normalizercon} hold.
We know by a theorem of Burnside (\cite{Linckelmann}, Thm. 1.2)
that if $S$ is abelian, then the normalizer controls fusion.
We explore this case in the final section.
At this point, we leave it as an open question whether or not
the normalizer must control fusion for all these properties to hold.

\section{Units from Maximal Subgroups}\label{MaxUnits}

Here we describe particular units of $B(S)^\times$ indexed by the maximal subgroups of $S$.
These units are constructed in such a way
that it is easy tell if they are contained in $B(\mathcal F)^\times$ for any $\mathcal F$ on $S$.

\begin{proposition}\label{maxunits}
Let $S$ be a 2-group.
For every maximal subgroup $M<S$,
there exists a unit $v_M\in B(S)^\times$
such that for $P\leq S,$ one has $|v_M^P|=-1$ iff $P\leq M$.
Moreover, $\mathrm{Aut}(S)$ acts on the set $\{v_M:M<S\text{ maximal}\}$
via $\alpha\cdot v_{M}=v_{\alpha(M)}$.
\end{proposition}
\begin{proof}
For any maximal $M<S,$ we have $S/M\cong C_2$.
Let $u_{S/M}\in B(S/M)^\times$ be defined by $u_{S/M}=[(S/M)/(S/M)]-[(S/M)/(M/M)]$.
Then $|u_{S/M}^{S/M}|=1$ and $|u_{S/M}^{M/M}|=-1$.
We set $v_M:=B(\mathrm{Inf}^S_{S/M})(u_{S/M})$.
Since $B(\mathrm{Inf}^S_{S/M})$ is a ring morphism, we have $v_M\in B(S)^\times$.
Moreover, $\phi_S(v_M)=\tilde B(\mathrm{Inf}^S_{S/M})(\phi_{S/M}(u_{S/M}))$.
Hence for any $P\leq S$, we have
$$|v_M^P|=|u_{S/M}^{PM/M}|=\begin{cases}1,&PM=S\\-1,&PM=M\end{cases}.$$
Notice since $M$ is maximal, $PM=M$ iff $P\leq M$,
and we have the first part of the proposition.

Notice that equivalently since $[S:M]=2,$
we have $\phi_S(v_M)=((-1)^{[PM:M]})_{P\leq S}$.
Now if $\alpha\in\mathrm{Aut}(S)$ and $M<S$ is maximal,
then also $\alpha(M)<S$ is maximal, and we also have $v_{\alpha(M)}\in B(S)^\times$.
We then see that
\begin{align*}
\phi_S(\alpha\cdot v_M)
&=\alpha\cdot\phi_S(v_M)
=\alpha\cdot((-1)^{[PM:M]})_{P\leq S}\\
&=((-1)^{[\alpha^{-1}(P)M:M]})_{P\leq S}
=((-1)^{[P\alpha(M):\alpha(M)]})_{P\leq S}\\
&=\phi_S(v_{\alpha(M)})
\end{align*}
Then since $\phi_S$ is injective, this shows that indeed $\alpha\cdot v_M=v_{\alpha(M)}$.
\end{proof}

We then use these ``maximal units" to uncover fusion properties of the corresponding maximal subgroups.

\begin{proposition}\label{stableclosed}
Let $\mathcal F$ be a saturated fusion system on a 2-group $S$,
and let $M<S$ be a maximal subgroup.
Then $v_M\in B(\mathcal F)^\times$ iff $M$ is strongly closed in $\mathcal F$.
\end{proposition}
\begin{proof}
First assume that $v_M\in B(\mathcal F)^\times$.
Let $x\in M$ and $y\in x^{\mathcal F}$.
Then there exists a $\varphi\in\mathrm{Iso}_{\mathcal F}(\langle x\rangle,\langle y\rangle)$
such that $\varphi(x)=y$.
Since $x\in M$, we have $|v_M^{\langle x\rangle}|=-1$.
And since $v_m\in B(\mathcal F)^\times,$ we have $|v_M^{\langle x\rangle}|=|v_M^{\langle y\rangle}|$.
Hence $|v_M^{\langle y\rangle}|=-1$ as well, and thus $y\in M$.
So we have shown $x$ is not $\mathcal F$-conjugate to any element of $S$ outside of $M$.
Therefore $M$ is strongly closed in $\mathcal F$.

Conversely, assume that $M$ is strongly closed in $\mathcal F$.
Let $P\leq S$ and $\varphi\in\mathrm{Hom}_{\mathcal F}(P,S)$.
If $P\leq M$, then $|v_M^P|=-1$.
And for all $x\in P,$ we must have $\varphi(x)\in M$ since $M$ is strongly closed.
Hence $\varphi(P)\leq M$ as well, and we have $|v_M^{\varphi(P)}|=-1=|v_M^P|$.
On the other hand, if $P\nleq M,$ then $|v_M^P|=1$.
Also there exists an $y\in P$ such that $y\notin M$.
And $\varphi(y)\notin M$ either, since $M$ is strongly closed.
Hence $\varphi(P)\nleq M$, and $|v_M^{\varphi(P)}|=1=|v_M^P|$.
In either case, we have $|v_M^P|=|v_M^{\varphi(P)}|$
for all $P\leq S$ and $\varphi\in\mathrm{Hom}_{\mathcal F}(P,S)$.
Therefore $v_M\in B(\mathcal F)^\times$.
\end{proof}

We can then use the fact that if a subgroup of $S$ is abelian,
then being strongly closed in $\mathcal F$
is equivalent to being normal in $\mathcal F$.

\begin{corollary}\label{maxstable}
Let $\mathcal F$ be a saturated fusion system on a 2-group $S$,
and suppose that $M$ is a maximal subgroup of $S$ that is abelian.
Then $v_M\in B(\mathcal F)^\times$ iff $M\unlhd\mathcal F$.
In this case,
the only possible $\mathcal F$-essential subgroups are $M$ and/or $S$.
\end{corollary}
\begin{proof}
By the proposition $v_M\in B(\mathcal F)^\times$ iff $M$ is strongly closed in $\mathcal F$.
By Corollary 4.7 in \cite{AKO}, this is equivalent to $M\unlhd\mathcal F$
given that $M$ is additionally abelian.
Then by Proposition 4.5 of that article,
$M$ must be contained in every $\mathcal F$-essential subgroups of $S$.
Since $M$ is maximal, $M$ and $S$ are the only possible essential subgroups.
\end{proof}

\begin{corollary}
Let $\mathcal F$ be a saturated fusion system on a 2-group $S$.
If $v_M\in B(\mathcal F)^\times$ for all maximal $M<S$,
then the Frattini subgroup $\Phi(S)$ is strongly closed in $\mathcal F$.
Additionally, if $\Phi(S)$ is abelian, then $\Phi(S)\unlhd\mathcal F$.
\end{corollary}
\begin{proof}
Assume $v_M\in B(\mathcal F)^\times$ for all maximal $M<S$.
Then each maximal subgroup is strongly closed in $\mathcal F$ by the proposition.
If $x\in\Phi(S)$,
then $x\in M$ for all maximal $M$.
Since $M$ is strongly closed, $x^{\mathcal F}\subseteq M$.
Hence $x^{\mathcal F}\subseteq\Phi(S)$,
showing that $\Phi(S)$ is strongly closed in $\mathcal F$.
The last statement follows from Corollary 4.7 of \cite{AKO}.
\end{proof}

\section{The Abelian Case}\label{Abelian}

In this section, we assume that $S$ is an abelian 2-group
and that $\mathcal F$ is a saturated fusion system on $S$.
We then know a nice basis of $B(S)^\times$ to start.
We state the result of Theorem 8.5 from \cite{Bouc} in this special case.

\begin{proposition}\label{abbasis}
Let $S$ be an abelian 2-group.
Then $B(S)^\times$ has an $\mathbb F_2$-basis given by $\{-1\}\cup\{v_M:M<S\text{ maximal}\}$.
\end{proposition}

We then use this basis of $B(S)^\times$ to create a basis of $B(\mathcal F)^\times$
for any saturated fusion system $\mathcal F$ on $S$.
This allows us to compute $\dim_{\mathbb F_2}B(\mathcal F)^\times$.

\begin{theorem}\label{abrank}
Let $\mathcal F$ be a saturated fusion system on an abelian groups $S$.
Then $\dim_{\mathbb F_2}B(\mathcal F)^\times=s+1$,
where $s$ is the number of $\mathcal F$-conjugacy classes of maximal subgroups of $S$.
\end{theorem}
\begin{proof}
Since $S$ is abelian, $S$ has no proper $\mathcal F$-essential subgroups.
By Theorem \ref{kernels}, we see that
$B(\mathcal F)^\times=\mathrm{Ker}(d^{\mathcal F}_S)=(B(S)^\times)^{\mathrm{Out}_{\mathcal F}(S)}$.
By Proposition \ref{abbasis}, we have a basis $\{-1\}\cup\{v_M:M<S\text{ maximal}\}$ of $B(S)^\times$.
Then Proposition \ref{splitsurj} shows that
$\mathrm{tr}^{\mathcal F}_S:B(S)^\times\longrightarrow(B(S)^\times)^{\mathrm{Out}_{\mathcal F}(S)}$
is split surjective.
But since $B(\mathcal F)^\times=(B(S)^\times)^{\mathrm{Out}_{\mathcal F}(S)},$
this means that $\{\mathrm{tr}^{\mathcal F}_S(-1)\}\cup\{\mathrm{tr}^{\mathcal F}_S(v_M):M<S\text{ maximal}\}$ is a spanning set of $B(\mathcal F)^\times$.
Clearly $-1\in B(\mathcal F)^\times$ and $\mathrm{tr}^{\mathcal F}_S(-1)=-1$.
And by Proposition \ref{maxunits},
we know that $\mathrm{Aut}_{\mathcal F}(S)$ acts on the set of the $v_M$.
Hence so does $\mathrm{Out}_{\mathcal F}(S)$
via $\bar\alpha\cdot v_M=v_{\alpha(M)}$.
So if $M<S$ is a maximal subgroup of $S$
and we denote the stabilizer of $v_M$ by $\mathrm{Out}_{\mathcal F}(S)_M$,
we note that $2\nmid|\mathrm{Out}_{\mathcal F}(S)_M|,$
and so
\begin{align*}
\mathrm{tr}^{\mathcal F}_S(v_M)
&=\prod_{\bar\alpha\in\mathrm{Out}_{\mathcal F}(S)}(\bar\alpha\cdot v_M)
=\prod_{\bar\alpha\in\mathrm{Out}_{\mathcal F}(S)}v_{\alpha(M)}\\
&=\prod_{\bar\alpha\in\mathrm{Out}_{\mathcal F}(S)/\mathrm{Out}_{\mathcal F}(S)_M}v_{\alpha(M)}^{|\mathrm{Out}_{\mathcal F}(S)_M|}
=\prod_{\bar\alpha\in\mathrm{Out}_{\mathcal F}(S)/\mathrm{Out}_{\mathcal F}(S)_M}v_{\alpha(M)}\\
&=\prod_{N\in M^{\mathcal F}}v_N
\end{align*}
This implies then that if $M_1,\dots,M_s$ is a set of representatives
of the $\mathcal F$-conjugacy classes of maximal subgroups of $S$,
we have a basis of $B(\mathcal F)^\times$ given by
$\{-1,\mathrm{tr}^{\mathcal F}_S(v_{M_1}),\dots,\mathrm{tr}^{\mathcal F}_S(v_{M_s})\}$.
Thus $\dim_{\mathbb F_2}B(\mathcal F)^\times=s+1$ as claimed.
\end{proof}

When we have a saturated fusion system $\mathcal F$ on an abelian group $S$,
we can then detect if $\mathcal F$ is trivial just by computing $B(\mathcal F)^\times$.
This may seem somewhat circular at first,
but computing $B(\mathcal F)^\times$ using GAP, for instance,
can be simple when $\mathcal F$ is a Frobenius fusion system
given by a large finite group with small Sylow $p$-subgroup
and the $\mathcal F$-automorphism groups are not clear.

\begin{corollary}\label{abtriv}
Let $S$ be an abelian 2-group and $\mathcal F$ a saturated fusion system on $S$.
Then $B(\mathcal F)^\times=B(S)^\times$ iff $\mathcal F=\mathcal F_S(S)$.
\end{corollary}
\begin{proof}
That $\mathcal F=\mathcal F_S(S)$ implies $B(\mathcal F)^\times=B(S)^\times$ is obvious.
Let us assume that $\mathcal F$ is not the trivial fusion system on $S$.
Then by \cite{Linckelmann} Theorem 3.11,
this means there exists $P\leq S$ such that $\mathrm{Aut}_{\mathcal F}(P)$ is not a 2-group.
Let us pick $\varphi\in\mathrm{Aut}_{\mathcal F}(P)$ of odd order.
Also by Theorem 3.8 of the same article, since $S$ is abelian, $S\unlhd F$.
Hence $\varphi$ extends to a $\psi\in\mathrm{Aut}_{\mathcal F}(S)$.
We can also assume that $\psi$ has odd order.
Let $\bar\psi$ denote the image of $\psi$ under the map
$\mathrm{Aut}(S)\rightarrow\mathrm{Aut}(S/\Phi(S))$.
By a result of Burnside (see Theorem 5.1.14 of \cite{Gorenstein}, for instance),
we know that the kernel of this map is a 2-group.
Thus $\bar\psi$ is nontrivial of odd order.
Then since $S/\Phi(S)$ is an elementary abelian 2-group,
there exists distinct maximal subgroups $M/\Phi(S),N/\Phi(S)$ of $S/\Phi(S)$
such that $\bar\psi(M/\Phi(S))=N/\Phi(S)$.
Then $M$ and $N$ are distinct maximal subgroups of $S$ such that $\psi(M)=N$.
Thus neither are strongly closed in $\mathcal F$,
and by Proposition \ref{stableclosed},
we have $v_M,v_N\notin B(\mathcal F)^\times$.
Hence $B(\mathcal F)^\times<B(S)^\times$.
\end{proof}

Notice when $S$ is not abelian,
we cannot assume that $S\unlhd\mathcal F$.
So an odd order $\varphi\in\mathrm{Aut}_{\mathcal F}(P)$ might not extend to $S$ for some $P<S$.
So even though we can find maximal subgroups of $P$ fused by $\mathcal F$,
they may be contained in the same maximal subgroups of $S$,
and we cannot use Proposition \ref{stableclosed}
to find non-$\mathcal F$-stable units of $B(S)^\times$ in this way.
However, by Corollary \ref{maxstable},
we see that $S$ could have at most one abelian maximal subgroup.
So this limits the search for a counterexample to Corollary \ref{abtriv}
for a nonabelian 2-group,
but at this point, we do not know such a counterexample.


\end{document}